\documentclass{wscpaperproc}
\usepackage{latexsym}
\usepackage{graphicx}
\usepackage{mathptmx}

\usepackage{amsmath}
\usepackage{amsfonts}
\usepackage{amssymb}
\usepackage{amsbsy}
\usepackage{amsthm}
\newtheorem{algorithm}{Algorithm}
\usepackage{algorithmic}

\usepackage{amsmath, amssymb, bbm, xspace}

\newcommand{\Real}{\ensuremath{\mathbb{R}}}

\def\spose#1{\hbox to 0pt{#1\hss}}

\def\text #1{\hbox{\quad#1\quad}}

\def\nthinsp{\mskip -2   mu}

\def\superstar{^{\raise 0.5pt\hbox{$\nthinsp *$}}}
\def\SUPERSTAR{^{\raise 0.5pt\hbox{$*$}}}

\def\lamstarT {\lambda^{\raise 0.5pt\hbox{$\nthinsp *$}T}}

\def\Nscr{{\cal N}}

\def\Nscr{{\cal N}}

\def\hbar{\skew{4.2}\bar h}

\def\wbar{\skew3\bar w}

		\def\bkE{{\rm I\kern-.17em E}}
		\def\bk1{{\rm 1\kern-.17em l}}
		\def\bkD{{\rm I\kern-.17em D}}
		\def\bkR{{\rm I\kern-.17em R}}
		\def\bkP{{\rm I\kern-.17em P}}
		\def\bkY{{\bf \kern-.17em Y}}
		\def\bkZ{{\bf \kern-.17em Z}}

		\def\beq{\begin{eqnarray}}
		\def\bc{\begin{center}}
		\def\be{\begin{enumerate}}
		\def\bi{\begin{itemize}}
		\def\bs{\begin{small}}
		\def\bS{\begin{slide}}
		\def\ec{\end{center}}
		\def\ee{\end{enumerate}}
		\def\ei{\end{itemize}}
		\def\es{\end{small}}
		\def\eS{\end{slide}}
		\def\eeq{\end{eqnarray}}

	\def\cp2problem#1#2#3#4{\fbox
		 {\begin{tabular*}{0.9\textwidth}
			{@{}l@{\extracolsep{\fill}}l@{\extracolsep{6pt}}l@{\extracolsep{\fill}}c@{}}
				#1 & & $#4 $ 
			\end{tabular*}}}

		\renewcommand{\emph}[1]{\textbf{#1}}

		\def\bkE{{\rm I\kern-.17em E}}
		\def\bk1{{\rm 1\kern-.17em l}}
		\def\bkD{{\rm I\kern-.17em D}}
		\def\bkR{\mathbb{R}}
		\def\bkP{{\rm I\kern-.17em P}}
		
		\def\bkZ{{\bf{Z}}}

\newcommand {\beeq}[1]{\begin{equation}\label{#1}}
\newcommand {\eeeq}{\end{equation}}
\newcommand {\bea}{\begin{eqnarray}}
\newcommand {\eea}{\end{eqnarray}}

\def\texitem#1{\par\smallskip\noindent\hangindent 25pt
               \hbox to 25pt {\hss #1 ~}\ignorespaces}

\usepackage{cite}

\usepackage[pdftex,colorlinks=true,urlcolor=blue,citecolor=black,anchorcolor=black,linkcolor=black]{hyperref}

\newtheoremstyle{wsc}
{3pt}
{3pt}
{}
{}
{\bf}
{}
{.5em}
{}
\usepackage{amsmath,amssymb}
\theoremstyle{wsc}
\newtheorem{theorem}{Theorem}

\newtheorem{lemma}{Lemma}
\newtheorem{remark}{Remark}

\newtheorem{assumption}{Assumption}

\def\cA{\mathcal A} 
\newcommand{\aj }[1]{{\color{black}#1}}
\newcommand{\us}[1]{{\color{black}#1}}
\newcommand{\uvs}[1]{{\color{black}#1}}
\newcommand{\vvs}[1]{{\color{black}#1}}
\newcommand{\af}[1]{{\color{black}#1}}
\newcommand{\afj}[1]{{\color{black}#1}}
\begin{document}

%
%

\pagestyle{fancyplain}

\thispagestyle{plain}
\firstPageHead{}

\chead{\fancyplain{}{\itshape Jalilzadeh and Shanbhag}}

\rhead{}
\cfoot{}
\renewcommand{\headrulewidth}{0pt} 

\makeatletter
\let\@internalcite\cite
\def\cite{\def\@citeseppen{-1000}%
    \def\@cite##1##2{(##1\if@tempswa , ##2\fi)}%
    \def\citeauthoryear##1##2##3{##1 ##3}\@internalcite}
\def\citeNP{\def\@citeseppen{-1000}%
    \def\@cite##1##2{##1\if@tempswa , ##2\fi}%
    \def\citeauthoryear##1##2##3{##1 ##3}\@internalcite}
\def\citeN{\def\@citeseppen{-1000}%
    \def\@cite##1##2{##1\if@tempswa, ##2)\else{}\fi}%
    \def\citeauthoryear##1##2##3{##1 (##3)}\@citedata}
\def\citeA{\def\@citeseppen{-1000}%
    \def\@cite##1##2{(##1\if@tempswa , ##2\fi)}%
    \def\citeauthoryear##1##2##3{##1}\@internalcite}
\def\citeANP{\def\@citeseppen{-1000}%
    \def\@cite##1##2{##1\if@tempswa , ##2\fi}%
    \def\citeauthoryear##1##2##3{##1}\@internalcite}
\def\shortcite{\def\@citeseppen{-1000}%
    \def\@cite##1##2{(##1\if@tempswa , ##2\fi)}%
    \def\citeauthoryear##1##2##3{##2 ##3}\@internalcite}
\def\shortciteNP{\def\@citeseppen{-1000}%
    \def\@cite##1##2{##1\if@tempswa , ##2\fi}%
    \def\citeauthoryear##1##2##3{##2 ##3}\@internalcite}
\def\shortciteN{\def\@citeseppen{-1000}%
    \def\@cite##1##2{##1\if@tempswa, ##2\else{}\fi}%
    \def\citeauthoryear##1##2##3{##2 (##3)}\@citedata}
\def\shortciteA{\def\@citeseppen{-1000}%
    \def\@cite##1##2{(##1\if@tempswa , ##2\fi)}%
    \def\citeauthoryear##1##2##3{##2}\@internalcite}
\def\shortciteANP{\def\@citeseppen{-1000}%
    \def\@cite##1##2{##1\if@tempswa , ##2\fi}%
    \def\citeauthoryear##1##2##3{##2}\@internalcite}
\def\citeyear{\def\@citeseppen{-1000}%
    \def\@cite##1##2{(##1\if@tempswa , ##2\fi)}%
    \def\citeauthoryear##1##2##3{##3}\@citedata}
\def\citeyearNP{\def\@citeseppen{-1000}%
    \def\@cite##1##2{##1\if@tempswa , ##2\fi}%
    \def\citeauthoryear##1##2##3{##3}\@citedata}
%
%
%
\def\@citedata{%
    \@ifnextchar [{\@tempswatrue\@citedatax}%
                  {\@tempswafalse\@citedatax[]}%
}

\def\@citedatax[#1]#2{%
\if@filesw\immediate\write\@auxout{\string\citation{#2}}\fi%
  \def\@citea{}\@cite{\@for\@citeb:=#2\do%
    {\@citea\def\@citea{, }\@ifundefined
       {b@\@citeb}{{\bf ?}%
       \@warning{Citation `\@citeb' on page \thepage \space undefined}}%
{\csname b@\@citeb\endcsname}}}{#1}}%

%
\def\@citex[#1]#2{%
\if@filesw\immediate\write\@auxout{\string\citation{#2}}\fi%
  \def\@citea{}\@cite{\@for\@citeb:=#2\do%
    {\@citea\def\@citea{; }\@ifundefined
       {b@\@citeb}{{\bf ?}%
       \@warning{Citation `\@citeb' on page \thepage \space undefined}}%
{\csname b@\@citeb\endcsname}}}{#1}}%

%
\def\@biblabel#1{}
\makeatother



\newdimen\bibindent
\bibindent=0.0em
\def\thebibliography#1{\section*{\refname}\list
   {}{\settowidth\labelwidth{[#1]}
   \leftmargin\parindent
   \itemindent -\parindent
   \listparindent \itemindent
   \itemsep 0pt
   \parsep 0pt}
   \def\newblock{}
   \sloppy
   \sfcode`\.=1000\relax}


\setlength{\baselineskip}{12.7pt}

\title{A PROXIMAL-POINT ALGORITHM WITH VARIABLE SAMPLE-SIZES ({\bf PPAWSS}) \\ FOR MONOTONE STOCHASTIC VARIATIONAL INEQUALITY PROBLEMS}

\author{Afrooz Jalilzadeh \\ Uday V. Shanbhag\\ [12pt]
Department of Industrial and Manufacturing Engineering, \\Pennsylvania State University\\
 University Park, PA 16802, USA.
}

\maketitle

\section*{ABSTRACT}
We consider a stochastic variational inequality (SVI) problem with a continuous
and monotone mapping over a closed and convex set.  In strongly monotone
regimes, we present a \uvs{variable sample-size averaging scheme} ({\bf
VS-Ave}) that achieves a linear rate with an optimal oracle complexity. \afj{In
addition, the iteration complexity is shown to display a muted dependence on
the condition number compared with standard variance-reduced projection
schemes.} 
To contend with merely monotone maps, we develop
amongst the first \uvs{proximal-point algorithms with variable
sample-sizes ({\bf PPAWSS})}, where increasingly accurate solutions of
strongly monotone SVIs are obtained via ({\bf VS-Ave}) at every step. This
allows for achieving \afj{a sublinear convergence rate} \uvs{that
matches that obtained} for deterministic monotone VIs. Preliminary numerical
evidence suggests that the schemes compares well with competing schemes. 
\section{INTRODUCTION}
Variational inequality problems have a broad range of
applications in engineering, economics, and the applied
sciences. Amongst the most common instances of where such
problems assume relevance is in the minimization of a
differentiable convex function $f$ over a closed and convex
set. Recall that $x$ is a minimizer if and only if  $(y-x)^T
\nabla_x f(x) \geq 0$ for all $y \in X,$ a variational
inequality problem succinctly denoted by VI$(X,\nabla_x
f)$. More generally, the variational inequality problem VI$(X,F)$
requires an $x \in X$ such that 
\begin{align} \label{VI$(X,F)$}
(y-x)^T F(x) \geq 0, \quad \forall y \in X. \end{align}
\vvs{In addition}, such problems find
application in modeling saddle-point problems, convex Nash games, traffic equilibrium problems,
economic equilibrium problems, amongst others
(see~\citeN{facchinei2007finite}).  While deterministic
variants~\cite{facchinei2007finite} have seen received
significant study over the last several decades, less is known
regarding the stochastic variant where
the map is replaced by its expectation-valued \vvs{counterpart}. \vvs{Such a problem}, denoted by SVI$(X,F)$, requires finding an $x\in X$ such that
\begin{align}\label{VI}
(y-x)^TF(x)\geq 0, \quad \forall y\in X, \end{align}
where $F(x)\triangleq \mathbb E[G(x,\xi)]$, $\xi: \Omega \to \Real^d$, ${G}: X \times \Real^d  \rightarrow
\mathbb{R}^n$, and the associated probability space is denoted by $(\Omega, {\cal F}, \mathbb{P})$.

Amongst the earliest schemes for resolving SVI$(X,F)$ via stochastic
approximation was \vvs{presented} by~\citeN{jiang08stochastic}
for strongly monotone maps. Regularized variants were developed
by~\shortciteN{koshal13regularized} for merely monotone
regimes while Lipschitzian requirements were weakened by
combining local smoothing with regularization
by~\shortciteN{yousefian2013regularized} and~\shortciteN{yousefian17smoothing}. \vvs{A rate statement of $\mathcal{O}(1/\sqrt{k})$}  for
the merely monotone regime (\vvs{in terms of a suitably defined
gap function defined in \eqref{gap-fn}}) was first provided by
~\shortciteN{juditsky2011solving}, {where $k$ \vvs{denotes} the number
of iterations}; this appears to have been the best known rate
for monotone stochastic variational inequality problems.

\vvs{When} sampling is ``cheap'', \vvs{
recent} efforts \uvs{have attempted to} reduce the gap between the rates of
convergence between deterministic schemes and their stochastic
counterparts. In particular, by increasing the sample-size
associated with \vvs{approximating} the gradient at a suitable
pace, the rate may be improved. In the context
of strongly convex stochastic optimization, linear rates of
convergence have been proven in both smooth 
(see~\citeN{shanbhag15budget} and~\citeN{jofre2017variance}),
\vvs{and a subclass of nonsmooth settings} (see 
~\shortciteN{jalilzadeh2018optimal}). Optimal rate statements of
$\mathcal{O}(1/k^2)$ via accelerated variance-reduced schemes
in stochastic convex differentiable settings were developed by \vvs{~\citeN{ghadimi2013stochastic} and} 
~\citeN{jofre2017variance} matching the best known deterministic
rate.
Similarly, a rate of $\mathcal{O}(1/k)$ was provided for nonsmooth
but smoothable \vvs{stochastic convex optimization problems} by~\shortciteN{jalilzadeh2018optimal}. To the best
of our knowledge, the only variance-reduced scheme for  monotone
stochastic \vvs{variational inequality problems} was presented by~\shortciteN{iusem2017extragradient}
and is equipped with a rate of $\mathcal{O}(1/k)$, matching the
deterministic rate for monotone variational inequality
problems. \vvs{It is worth emphasizing that the monotone
stochastic variational inequality problems can capture
stochastic convex optimization problems (with possibly
expectation-valued problems), stochastic convex Nash games, and
a range of stochastic equilibrium problems. In fact, stochastic
convex optimization problems can be solved via acceleration and
variance reduction~\cite{jofre2017variance}. However,
variational inequality problems are generally not equipped with
an appropriate with a natural ``objective'' function but the 
gap function (which is generally nonsmooth) serves as a metric for measuring progress of a
scheme.  The best known rate for
monotone deterministic variational inequality problems is
$\mathcal{O}(1/k)$ (in terms of a gap function); recall that a
monotone variational inequality problem can be recast as
nonsmooth stochastic convex problem (for which the best known
rate under smoothability requirements is $\mathcal{O}(1/k)$ in
function values)}.      

 {\bf (I). Linearly convergent schemes for strongly
monotone SVIs.} When {$F(x)$ is strongly monotone {over
$X$}, i.e. $F(x)-F(y), x-y\rangle\geq \mu \|x-y\|^2$ for all
$x,y\in X$, {the unique} solution $x^*$ of \eqref{VI}
satisfies {the following} \us{for all $y \in X$.}
\begin{align*}
\langle F(y), x^*-y\rangle + {1\over 2}\mu \|y-x^*\|^2  \overset{\tiny [\mbox{Strong. monot}]}{\leq} \overbrace{\langle F(x^*), x^*-y\rangle}^{\ \leq \ 0} - {1\over 2}\mu \|y-x^*\|^2 \overset{\tiny [x^* \in \mbox{SOL}(X,F)]}{\leq} 0.
\end{align*}
Recall that a gap function associated with  \eqref{VI$(X,F)$}, denoted
by  $g:\mathbb R^n\rightarrow \mathbb R$, is a nonnegative
function on $X$ and $g(x^*)=0$ for $x^*\in \mathbb R^n$ if and
only if $x^*$ is a solution of SVI. Therefore, solving
SVI$(X,F)$ is equivalent to minimizing a gap function on
$X$. Inspired by \citeN{nesterov2011solving}, we \vvs{define $g(x)$ as follows.} \begin{align} \label{gap-fn} g(x) \triangleq \sup_{y\in X}\left\{
\langle F(y), x-y\rangle+{1\over 2}\mu \|y-x\|^2 \right\}.\end{align}}
\citeN{nesterov2011solving} proposed a linearly
convergent \vvs{averaging scheme that is reliant on two projections at every step, the first of which \uvs{is a projection of} an averaged vector while the second is the more standard projection.} \vvs{We apply this avenue with two key distinctions: (i) We utilize a noise-corrupted estimate of the expectation-valued map; and (ii) variance reduction is employed to reduce the bias as $k$, the iteration counter. In Section \ref{SVI sec.}, we show that the resulting \uvs{variable sample-size averaging ({\bf VS-Ave})} scheme admits a linear rate of convergence when the sample-size is increased at a
geometric rate.} Importantly, the scheme
enjoys a muted dependence on the condition number of the map,
in contrast with standard variance-reduced projection schemes \vvs{and is characterized by an oracle complexity of $\mathcal{O}(1/\epsilon^{\beta})$ (where $\beta > 1$) for computing an
$\epsilon$-solution.}

 {\bf (II).  Stochastic proximal schemes for monotone SVIs.} In
section \ref{SSP sec.}, we develop a stochastic generalization
of the proximal-point scheme, referred to as a \uvs{proximal-point algorithm with variable sample-sizes ({\bf PPAWSS})} \vvs{for monotone SVIs} \vvs{where} the \vvs{strongly monotone} proximal subproblems
are solved with increasing \vvs{accuracy}
via ({\bf VS-Ave}). \vvs{We} show
that the sequence of iterates converges to the solution at the
rate of $\mathcal{O}(1/k)$, matching the deterministic rate  while the iteration complexity in inner projection
steps is bounded as $\mathcal O\left(\tfrac{1}{\epsilon}\log\tfrac{1}{\epsilon}\right)$. {Preliminary numerics in
Section 5 suggest that the scheme compares well with competing
techniques.} 

\section{LINEARLY CONVERGENT SCHEMES FOR STRONGLY MONOTONE SVI PROBLEMS}\label{SVI sec.}
In this section, inspired by \citeN{nesterov2011solving}, we propose a
scheme for solving strongly monotone SVIs which admits a linear rate of convergence with a 
constant of {${L/\mu}$} in the iteration complexity
bound, significantly smaller than {$(L/\mu)^2$}, obtained from
standard extensions of projection methods. Our rate statements
rely on the following assumptions on $F$.
\begin{assumption}\label{assump-F}
Assume that operator $F:X\rightarrow \mathbb R^n$ is $\mu$-strongly monotone and $L$-
Lipschitz continuous on $X$, i.e. $\|F(x)-F(y)\|\leq L\|x-y\|$ \us{for all $x, y \in X$.}
\end{assumption}
Next, if $\mathcal F_{k}$ denotes the information history at epoch $k$, then we have the following requirements on the associated filtrations \vvs{where $\wbar_{k,N_k} \triangleq \tfrac{1}{N_k}{\sum_{j=1}^{N_k} (F(y_k) - G(y_k,\omega_{j,k}))}$ and $\wbar'_{k,N_k} \triangleq \tfrac{1}{N_k} {\sum_{j=1}^{N_k} (F(x_k) - G(x_k,\omega_{j,k}))}.$}
\begin{assumption}\label{assump_error}
 There exists $\nu,\tilde \nu>0$ such that $\mathbb E[\wbar_{k+1,\vvs{N_{k+1}}}\mid  \mathcal F_{k}\cup \mathcal F'_{k-1}]=0$, $\mathbb E[ \wbar'_{k,\vvs{N_k}}\mid  \mathcal F_{k}\cup \mathcal F'_{k-1}]=0$, $\mathbb E[\| \wbar_{k+1,\vvs{N_{k+1}}}\|^2\mid  \mathcal F_{k}\cup \mathcal F'_{k-1}]\leq \tfrac{\nu^2}{\vvs{N_{k+1}}}$ and $\mathbb E[\| \vvs{\wbar}'_{k,\vvs{N_k}}\|^2\mid  \mathcal F_{k}\cup \mathcal F'_{k-1}]\leq \tfrac{\tilde \nu^2}{\vvs{N_k}}$ holds almost surely 
 for all $k$, 
where $\mathcal{F}_k \triangleq \sigma\{y_0, y_1, \hdots, y_{k-1}\}$, $w'_k \triangleq F(x_k)- G(x_k,\xi_k)$ and  $\mathcal{F'}_k \triangleq \sigma\{x_0, x_1, \hdots, x_{k-1}\}$.  \end{assumption} Throughout the paper, we exploit the following basic lemma stated without proof.  
\begin{lemma}\label{3norms}
Given a symmetric positive definite matrix $Q$, we have the following for any $\nu_1,\nu_2,\nu_3$:
\begin{align*}(\nu_2-\nu_1)^TQ(\nu_3-\nu_1)={1\over 2}(\|\nu_2-\nu_1\|^2_Q+\|\nu_3-\nu_1\|^2_Q-\|\nu_2-\nu_3\|^2_Q), \mbox{ where } \|\nu\|_Q \triangleq \sqrt{\nu^TQ\nu}.\end{align*}
\end{lemma}
\vvs{Before proceeding, we recall} some properties of gap function $g(x)=\sup_{y\in X}\left\{ \langle F(y), x-y\rangle+{1\over 2}\mu \|y-x\|^2 \right\}$. 
\begin{theorem}{ \cite{nesterov2011solving}}
The gap function $g(x)$ is well defined and $\mu$-strongly convex on $X.$ Moreover, $g(x)$ is nonnegative on $X$ and vanishes only at the unique solution of~\eqref{VI}. 
\end{theorem}
We present an averaging-based variance-reduced ({\bf VS-Ave}) scheme (Alg.~\ref{SVI alg.}) for  strongly monotone SVIs where $K$ and $\Pi_X(y)$ denote the total no. of iterations and  the projection of $y$ onto $X$, respectively. 
\\
\begin{algorithm}[{\bf VS-Ave}$(F,y_0,\mu,L,K$)] \em \label{SVI alg.}
\begin{enumerate}
\item[]
\item[(1.0)] Given $y_0\in X$ and let $\gamma_0=1$, $k = 0$; $\Gamma_0 = 1$; 
\item[(1.1)] $x_{k} = \Pi_{X}\left[{1\over \sum_{i=0}^k \gamma_i}\left(\sum_{i=0}^k \gamma_i\left(y_i-{1\over  \mu}{(F(y_i)+\bar w_{i,N_i})}\right)\right)\right];$ 
\item[(1.2)] $y_{k+1}=\Pi_{X}\left[x_k-(\tfrac{1}{ L}) {(F(x_k)+\bar w'_{k,N_k})}\right];$
\item[(1.3)] $\gamma_{k+1}={\mu\over \mu+ L}\Gamma_k; \Gamma_{k+1} = \Gamma_k + \gamma_{k+1};$
\item[(1.4)] If $k > K$, then stop; else $k = k+1$; return
to (1).
\item[(1.5)] Return $\bar y_k={1\over \Gamma_k}\sum_{i=0}^k \gamma_iy_i$.
\end{enumerate}
\end{algorithm}
\begin{remark}\label{proj.}{By definition of the projection in step (1.1) of Algorithm \ref{SVI alg.}, we have that 
$$\Big\langle x_k-{1\over \sum_{i=1}^k \gamma_i }\left(\sum_{i=1}^k \gamma_i \left[y_i-{1\over \mu} (F(y_i) {+\bar w_{i,N_i}})\right]\right),x_k-x\Big\rangle\leq0 \quad \mbox{for all} x\in X.$$ 
This can be equivalently represented as $\sum_{i=1}^k \gamma_i \langle F(y_i) {+\bar w_{i,N_i}}+ \mu (x_k-y_i),x_k-x\rangle\leq0$ for all $x\in X$, which \vvs{captures the optimality of $x_k$ with respect to} $\max_{x\in X} \left\{\sum_{i=0}^{k} \gamma_i\left[\langle F(y_i) {+\bar w_{i,N_i}},y_i-x\rangle-{1\over 2} \mu \|x-y_i\|^2\right]\right\}$}. 
\end{remark}
{In the following lemma}, we provide an upper bound for the gap function. 
\begin{lemma}{ \cite{nesterov2011solving}}\label{bound by delta}
Consider a sequence of positive weights $\{\gamma_i\}_{i=0}^K$ and points $\{y_i\}_{i=0}^K\subset\mathbb R^n$. Let $\Gamma_K=\sum_{i=0}^K \gamma_i$, $\bar y_K={1\over \us{\Gamma_k}}\sum_{i=0}^K \gamma_iy_i$ and $ {\tilde\cA_K}=\max_{x\in X}\left\{\sum_{i=0}^K \gamma_i\left[\langle F(y_i),y_i-x\rangle-{1\over 2} \mu \|x-y_i\|^2\right]\right\}$. Then, $g(\bar y_K)\leq{1\over \Gamma_K} {\tilde\cA_K}.$
\end{lemma}

Next, we derive a bound on the conditional expectation on ${\mathcal A}_{k+1}$.

\begin{lemma} \label{A_{k+1}<A_k}
Consider the sequences generated by Algorithm \ref{SVI alg.}. \vvs{Suppose} 
Assumptions \ref{assump-F} and \ref{assump_error} hold, and
$\cA_k \triangleq \max_{x\in X}\left\{\sum_{i=0}^k \gamma_i\left[\langle F(y_i)
{+\bar w_{i,N_i}},y_i-x\rangle-{1\over 2} \mu \|x-y_i\|^2\right]\right\}.$
Then \vvs{the following holds for all $k$.} $$\mathbb
E[\cA_{k+1}\mid \mathcal F_{k}\cup \mathcal F'_{k-1}]\leq \cA
_k+\gamma_{k+1}\left(\tfrac{1}{c}+\tfrac{1}{\mu}\right)(\tfrac{\nu^2}{N_{k+1}}+\tfrac{(\tilde \nu)^2}{N_{k}}), \mbox{ where }  c \triangleq \tfrac{L \mu}{L+ \mu}.$$
\end{lemma}
\begin{proof}
From the definition of $\cA_{k+1}$ we have that:
\begin{align}\label{def_delta}
\cA_{k+1}\nonumber&=\max_{x\in X}\bigg\{\sum_{i=0}^{k} \gamma_i\underbrace{\left[\langle F(y_i) {+\bar w_{i,N_i}},y_i-x\rangle-{1\over 2} \mu \|x-y_i\|^2\right]}_{\tiny {\mbox{Term  (A)}}} \\
&  +\gamma_{k+1}\underbrace{\left[\langle F(y_{k+1}) {+\bar w_{k+1,N_{k+1}}},y_{k+1}-x\rangle-{1\over 2} \mu \|x-y_{k+1}\|^2\right]}_{\tiny \mbox{Term  (B)}} \bigg\}.
\end{align}
\aj{One can easily check that term $(A)$ in \eqref{def_delta} is $\mu$-strongly concave \us{in $x$.} Therefore, we obtain the following:
\begin{align}\label{sconcave}
&\nonumber\sum_{i=0}^{k} \gamma_i\left[\langle F(y_i) {+\bar w_{i,N_i}},y_i-x\rangle-{\mu\over 2}  \|x-y_i\|^2\right]   =   \sum_{i=0}^{k} \gamma_i\left[\langle F(y_i) {+\bar w_{i,N_i}},y_i-x_k\rangle-{\mu \over 2}  \|x_k-y_i\|^2\right] \\ &+\sum_{i=0}^{k} \gamma_i\left[\langle F(y_i) {+\bar w_{i,N_i}}+ \mu(x_k-y_i),x_k-x\rangle-{ \mu\over 2}\|x-x_k\|^2\right] \leq\cA_k-{ \mu \Gamma_k\over 2}\|x-x_k\|^2, \quad\forall x\in X,
\end{align}
where in the last inequality, we used $\sum_{i=1}^k \gamma_i \langle F(y_i) {+\bar w_{i,N_i}}+ \mu (x_k-y_i),x_k-x\rangle\leq0$ for all $x\in X$ from Remark \ref{proj.} and  $\Gamma_k=\sum_{i=0}^{k} \gamma_i$.}  From the definition of $y_k$ in Algorithm~\ref{SVI alg.}, 
\begin{align}
\label{yk-opt}
\langle F(x_k) {+\bar w'_{k,N_k}}+\af{L} (y_{k+1}-x_k),y_{k+1}-x\rangle \leq 0  \mbox{ for all } x\in X.
\end{align} Consequently, Term (B) can be bounded as follows by using \eqref{yk-opt} and the Cauchy-Schwarz inequality:
\begin{align}\label{bound_proof}
\langle \nonumber&F(y_{k+1}) {+\bar w_{k+1,N_{k+1}}},y_{k+1}-x\rangle-{ \mu \over 2 }\|x-y_{k+1}\|^2\\ \nonumber&=\langle F(y_{k+1})-F(x_k),y_{k+1}-x\rangle-{ \mu \over 2 }\|x-y_{k+1}\|^2+\langle F(x_k) {+\bar w'_{k,N_k}},y_{k+1}-x\rangle {+\langle  {\bar w_{k+1,N_{k+1}}}- {\bar w'_{k,N_k}},y_{k+1}-x\rangle}\\ \nonumber&
\overset{\eqref{yk-opt}}{\leq} \|F(y_{k+1})-F(x_k)\|\|y_{k+1}-x\|-{ \mu \over 2 }\|x-y_{k+1}\|^2+\af{L} \left\langle (y_{k+1}-x_k),x-y_{k+1}\right\rangle\\& {+\langle  {\bar w_{k+1,N_{k+1}}}- {\bar w'_{k,N_k}},y_{k+1}-x\rangle}.
\end{align}
\us{Applying} Lemma \ref{3norms} \us{with} $Q = I$, we have that \us{the penultimate term can be expressed as follows.}
\begin{align}\label{3norms_p}
2\left\langle (y_{k+1}-x_k,x-y_{k+1})\right\rangle=\|x-x_k\|^2-\|y_{k+1}-x_k\|^2-\|x-y_{k+1}\|^2.
\end{align}
Hence, using \eqref{3norms_p} and the fact that $\us{a^Tb}\leq {1\over 2(L+\mu)}\|a\|^2+{L+\mu\over 2}\|b\|^2$, inequality \eqref{bound_proof} can be written as follows: 
\begin{align}\label{second_term}
\langle \nonumber&F(y_{k+1}) {+\bar w_{k+1,N_{k+1}}},y_{k+1}-x\rangle-{ \mu \over 2 }\|x-y_{k+1}\|^2\\ &\leq {1\over 2(L+\mu)}\|F(y_{k+1})-F(x_k)\|^2+{(L+\mu)\over 2} \|y_{k+1}-x\|^2 - {\mu \over 2}\|x-y_{k+1}\|^2 \notag \\
& \notag + {L \over 2}  \|x-x_k\|^2-{L\over 2} \|y_{k+1}-x_k\|^2 - {L \over 2}\|x-y_{k+1}\|^2 {+\langle  {\bar w_{k+1,N_{k+1}}}- {\bar w'_{k,N_k}},y_{k+1}-x\rangle} \\
& = {1\over 2(L+\mu)}\|F(y_{k+1})-F(x_k)\|^2+{L\over 2} \|x-x_k\|^2-{L\over 2} \|y_{k+1}-x_k\|^2 {+\langle  {\bar w_{k+1,N_{k+1}}}- {\bar w'_{k,N_k}},y_{k+1}-x\rangle}. 
\end{align}
\us{where the} {last term in \eqref{second_term} can be written as 
\begin{align}\label{bound w}
\langle  {\bar w_{k+1,N_{k+1}}}- {\bar w'_{k,N_k}},y_{k+1}-x_k+x_k-x\rangle\nonumber&\leq {1\over c}\left(\|\bar w_{k+1,N_{k+1}}\|^2+\| {\bar w'_{k,N_k}}\|^2\right)+{c\over 2}\|y_{k+1}-x_k\|^2\\&+{1\over \mu}\left(\|\bar w_{k+1,N_{k+1}}\|^2+\| {\bar w'_{k,N_k}}\|^2\right)+{\mu\over 2}\|x_{k}-x\|^2,
\end{align} 
where $c={L \mu\over L+ \mu}   < L$. {It follows that 
\vvs{by utilizing} \eqref{bound w} in  \eqref{second_term}  and then substituting the result and \eqref{sconcave}  in \eqref{def_delta} and by recalling that $\gamma_{k+1}=(\mu/ (L+\mu))\Gamma_k$, we get the following:
\begin{align}
\notag \cA_{k+1} & \leq \cA_{k} - {\mu \Gamma_k \over 2} \|x-x_k\|^2  + \gamma_{k+1} \left[ 
 {1\over 2(L+\mu)}\|F(y_{k+1})-F(x_k)\|^2+{L+\mu\over 2} \|x-x_k\|^2-{L-c\over 2} \|y_{k+1}-x_k\|^2\right.\\ \notag
	& \left. + \left({1\over c}+{1\over \mu}\right)\left(\|\bar w_{k+1,N_{k+1}}\|^2+\| {\bar w'_{k,N_k}}\|^2\right)\right] \\
\label{big bound} \notag
& \leq \cA_k+{\gamma_{k+1}\over 2}\left[{1\over L+\mu}\|F(y_{k+1})-F(x_k)\|^2-(L-c) \|y_{k+1}-x_k\|^2\right]\\
	&   +  \gamma_{k+1}\left({1\over c}+{1\over \mu}\right)\left(\|\bar w_{k+1,N_{k+1}}\|^2+\| {\bar w'_{k,N_k}}\|^2\right).
\end{align}}
  Now using Assumption \ref{assump-F}, we obtain the following.
\begin{align}\label{zero bound}
{1\over L+\mu}\|F(y_{k+1})-F(x_k)\|^2- \|y_{k+1}-x_k\|^2_{L I-cI}&\leq \|y_{k+1}-x_k\|^2_{({L^2\over L+\mu}-L+c){I} }\af{=} 0,
\end{align}
where {the last {equality follows}} from ${L^2\over  \mu+L}+c= L$, \vvs{a consequence of the} definition of $c$. 
Using  \eqref{zero bound} and \eqref{bound w} within \eqref{big bound} and taking conditional \vvs{expectations} with respect to $\mathcal F_{k}\cup \mathcal F'_{k-1}$, we obtain 
\begin{align*}
\mathbb E[\cA_{k+1}\mid \mathcal F_{k}\cup \mathcal F'_{k-1}]&\leq \cA_k+\gamma_{k+1}\left({1\over c}+{1\over \mu}\right)\mathbb E[\|\bar w_{k+1,N_{k+1}}\|^2+\| {\bar w'_{k,N_k}}\|^2\mid  \mathcal F_{k}\cup \mathcal F'_{k-1}],
\end{align*}
which leads to the desired result by using Assumption \ref{assump_error}. }
\end{proof}
\vvs{We now derive rate and complexity statements for Algorithm \ref{SVI alg.}.}
\begin{theorem}[{Rate and Complexity Statement for ({\bf VS-Ave})}]\label{rate svi}
Suppose Assumptions \ref{assump-F} and \ref{assump_error} hold. Suppose $\cA_k$ is as defined in Lemma \ref{A_{k+1}<A_k} and $x^*$ denotes the solution of \eqref{VI}. Consider the iterates generated by Algorithm \ref{SVI alg.} \vvs{and} let $N_k \triangleq \lfloor \rho^{-k}\rfloor $ such that $\rho<1-{1\over \kappa+2}$. 

\noindent (i) \vvs{If} $\kappa \vvs{\triangleq} \tfrac{L}{\mu} $ and $c \triangleq \tfrac{L \mu}{L+ \mu}$, the following holds for all $K$.
\begin{align*}
\tfrac{\mu\mathbb E[\|\bar y_K-x^*\|^2]}{2}\leq \left(g(y_0){\kappa^2}+2\kappa (\nu^2+\tilde \nu^2)\left({1\over c}+{1\over \mu}\right)\left({\kappa+1\over (\kappa+2)(1-\rho)-1}\right)\right)\left(1-{1\over \kappa+2}\right)^K=Cq^K.
\end{align*}
\noindent (ii) \vvs{Suppose $\bar y_K$ is an $\epsilon$-solution such that $\mathbb
E[\|\bar y_K-x^*\|^2]\leq \epsilon$. If $\bar{C} = 2C/\mu$, then Algorithm \ref{SVI alg.} requires $\mathcal O(\kappa \log (\bar{C}/\epsilon))$ steps and $\mathcal O(1/\epsilon)^\beta$ evaluations where $\rho=q^\beta$ \vvs{and $\beta > 1$.}}  \end{theorem}
\begin{proof}
(i) From the strong convexity of $g$, we have that ${\mu\over 2}\|\bar y_k-x^*\|^2\leq g(\bar y)$. Note that from Lemma \ref{A_{k+1}<A_k}, the following holds for all $k$ by taking unconditional expectations.
\begin{align}
\nonumber&\mathbb E[\cA_{k+1}\mid \mathcal F_{k}\cup \mathcal F'_{k-1}]\leq \cA _k+{\gamma_{k+1}}\left({1\over c}+{1\over \mu}\right)(\tfrac{\nu^2}{N_{k+1}}+\tfrac{\tilde \nu^2}{N_{k}})\implies \mathbb E[\cA_{k+1}]\leq \us{\mathbb{E}[\cA _k]}+{\gamma_{k+1}}\left({1\over c}+{1\over \mu}\right)(\tfrac{\nu^2}{N_{k+1}}+\tfrac{\tilde \nu^2}{N_{k}})\\
& \implies \mathbb E[\cA_{K}]\leq \cA _0+\sum_{k=0}^{K-1}{\gamma_{k+1}}\left({1\over c}+{1\over \mu}\right)(\tfrac{\nu^2}{N_{k+1}}+\tfrac{\tilde \nu^2}{N_{k}}).  \label{bd-A}
\end{align}
Since $\Gamma_0=\gamma_0=1$, we have that $\Gamma_{k+1}=\Gamma_k+\gamma_{k+1}=\Gamma_k+{\mu\over \mu+L}\Gamma_k=\left(1+{1\over 1+\kappa}\right)\Gamma_k$. Therefore, using Lemma \ref{bound by delta} and the fact that $\max \mathbb E[.]\leq \mathbb E[\max(.)]$, the following holds for $\bar y_k={1\over \Gamma_k}\sum_{i=0}^k \gamma_iy_i$. 
\begin{align}\label{bound1}\mathbb E[g(\bar y_K)]\leq {\mathbb E[\tilde \cA_K]\over \Gamma_K}\leq {\mathbb E[\cA_K]\over \Gamma_K}\us{\overset{\eqref{bd-A}}{\leq}} \left(\cA_0+\sum_{k=0}^{K-1}{\gamma_{k+1}}\left({1\over c}+{1\over \mu}\right)(\nu^2/N_{k+1}+\tilde \nu^2/N_{k})\right)\left(1-{1\over \kappa+2}\right)^K,\end{align}
where $\tfrac{1}{\Gamma_K} = \left(\tfrac{\kappa+1}{\kappa+2}\right)^{K} = (1-\tfrac{1}{\kappa+2})^{K}$. Using the definition of $\cA_0$, we obtain:
\begin{align*}
\cA_0&=\max_{x\in X}\left\{\langle F(y_0)+F(x^*)-F(x^*),y_0-x\rangle-{1\over 2} \mu \|x-y_0\|^2\right\}\\
& = \max_{x \in x} \underbrace{\langle F(x^*),x^*-x\rangle}_{\tiny \ \leq \ 0 \mbox{ since } x^* \in \ \mbox{SOL}(X,F) } + \langle F(x^*), y_0-x^*\rangle + \max_{x\in X}\left\{\langle F(y_0)-F(x^*),y_0-x\rangle-{1\over 2} \mu \|x-y_0\|^2\right\} \\
	& \leq  
F(x^*), y_0-x^*\rangle +\underbrace{\max_{x\in X}\left\{\langle F(y_0)-F(x^*),y_0-x\rangle-{1\over 2} \mu \|x-y_0\|^2\right\}}_{\tiny \mbox{Term (C)} }.
\end{align*}
{Since $a^Tb\leq {1\over 2\mu}\|a\|^2+{\mu\over 2}\|b\|^2$ for $a, b \in \Real^n$, $\mbox{Term (C)}\leq \max_{x\in X}\Big\{{1\over 2\mu}\|F(y_0)-F(x^*)\|^2+{\mu\over 2}\|y_0-x\|^2-{\mu \over 2}\|x-y_0\|^2\big\}={1\over 2\mu}\|F(y_0)-F(x^*)\|^2$}.
By Lipschitz continuity of $F$ and the
definition of $g$, we have that  \begin{align}\label{bound2}
\cA_0\nonumber&\leq \langle F(x^*),y_0-x^*\rangle+{L^2\over 2 \mu}\|y_0-x^*\|^2\leq g(y_0)-{\mu\over 2}\|x^*-y_0\|^2+{L^2\over 2\mu}\|y_0-x^*\|^2\\&
{ \ = \ } g(y_0)+\|y_0-x^*\|^2{\mu\over 2}\left({L^2\over \mu^2}-1\right)\leq {g(y_0) + g(y_0)\left({L^2\over \mu^2}-1\right)}  =  {L^2\over \mu^2} g(y_0),
\end{align}
where the {penultimate} inequality {follows from the} strong convexity of $g(y_0)$ and  $g(x^*)=0$. Furthermore, using the  definition of $\gamma_k$, we have the following bounds.
\begin{align}\label{bound3}
\sum_{k=0}^{K-1}{\gamma_{k+1}}\nonumber&\left({1\over c}+{1\over \mu}\right)(\tfrac{\nu^2}{N_{k+1}}+\tfrac{\tilde \nu^2}{N_{k}})={\kappa}\left({1\over c}+{1\over \mu}\right)\sum_{k=0}^{K-1}\left(1+{1\over \kappa+1}\right)^k(\tfrac{\nu^2}{N_{k+1}}+\tfrac{\tilde \nu^2}{N_{k}})\\ \nonumber&\leq {\kappa}\left({1\over c}+{1\over \mu}\right)\sum_{k=0}^{K-1}\left(1+{1\over \kappa+1}\right)^k\left(\nu^2+\tilde \nu^2\over N_{k}\right)\\ \nonumber&\leq {2\kappa (\nu^2+\tilde \nu^2)}\left({1\over c}+{1\over \mu}\right)\sum_{k=0}^{K-1} {\rho^k\over \left(1-{1\over \kappa+2}\right)^k}\leq {2\kappa (\nu^2+\tilde \nu^2)}\left({1\over c}+{1\over \mu}\right){1\over 1-{\rho\over 1-{1\over \kappa+2}}}\\&={2\kappa (\nu^2+\tilde \nu^2)}\left({1\over c}+{1\over \mu}\right)\left({\kappa+1\over (\kappa+2)(1-\rho)-1}\right).
\end{align}
By substituting \eqref{bound2} and \eqref{bound3} in \eqref{bound1} and by recalling that ${\mu\over 2}\|\bar y_K-x^*\|^2\leq g(\bar y_K)$, the result follows.

\noindent (ii) To obtain $\epsilon$-solution we need $\bar C
q^K\leq \epsilon$, where $\bar C=2C/\mu$. Therefore, we need at
least $K \geq \log_{1/q}\bar C/\epsilon={\log \tilde
C/\epsilon\over \log 1/q}$ steps. \uvs{From the definition of
$q$, we have that $\log (1/q)=\log (1+{1\over
\kappa+1})\geq\tfrac{1}{2\kappa+2}$ by recalling that
$\log(1+x) \geq \tfrac{x}{2}$ for $x \in [0,1]$.} 
\vvs{Consequently, the iteration complexity is $\mathcal O(\kappa \log (\bar{C}/\epsilon))$ where $\bar{C}  \leq {\mathcal O}( \kappa^2(\nu^2+\tilde \nu^2))$.} \uvs{Similarly, the} oracle complexity can be computed as follows for $\rho=q^{\beta}$ and $\beta>1$:
\begin{align*}
\sum_{k=0}^{\log_{1/q}\bar C/\epsilon+1} \lfloor \rho^{-k}\rfloor=\frac{\rho^{-\log_{1/q}\bar C/\epsilon-2}-1}{1/\rho-1}\leq \frac{\rho^{-\log_{1/q}\bar C/\epsilon-1}}{1-\rho}={\bar C^{\beta}\over \rho-\rho^2}(1/\epsilon)^\beta. 
\end{align*}
\end{proof}
\begin{remark}
Standard {variance-reduced projection schemes} for (SVIs) \vvs{are characterized by an iteration complexity bound of $\mathcal{O}(\kappa^2 \log(\tfrac{1}{\epsilon}))$ while Algorithm~\ref{rate svi} has a bound given by $\mathcal{O}(\kappa \log(\tfrac{1}{\epsilon}))$.}
\end{remark}
\section{PROXIMAL SCHEMES FOR MONOTONE  STOCHASTIC PROBLEMS}\label{SSP sec.} 
{Monotone stochastic variational inequality problems have been
examined extensively in the literature. In monotone
regimes, the best known result via variance reduction was provided by~\shortciteN{iusem2017extragradient} where it a
suitably defined residual function provably diminishes at
{$\mathcal{O}(1/k)$} when a stochastic extragradient scheme is
employed together with a variance reduction framework. We develop a new avenue for resolving SVIs through a
stochastic proximal point framework.} We begin by summarizing
proximal point schemes for generalized equations in
Section~\ref{summ_prox} and then extend and analyze  \vvs{a stochastic generalization} in Section~\ref{svi_prox}.
\subsection{A Review of Proximal Point Schemes}\label{summ_prox}
Consider the solution of a generalized equation $0 \in T(u),$
where $T: \Real^n \rightarrow \Real^n$ is a set-valued maximal monotone operator.  
{A map $T$ is said to be $\mu$-strongly monotone if there exists a $\mu > 0$ such that 
$ (v-u)^T(y-x) \geq \af{\mu} \|y-x\|^2, \forall x,y, $
where $v \in T(y)$ and $u \in T(x)$. Furthermore, if $\mu = 0$, this reduces to mere monotonicity.  Recall that the graph of a monotone map is defined as 
$ \mbox{gph}(T) \triangleq \left\{(x,y) \mid y \in T(x)\right\}. $
Then $T$ is a {\em maximal monotone} operator if no monotone map $\Psi$ exists such that $\mbox{gph}(T) \subset \mbox{gph}(\Psi).$ For instance, the subdifferential map $\partial f$ of a closed, convex, and proper function is a maximal monotone map.}   
The classical proximal point algorithm can be traced back to the seminal paper by~\citeN{rockafellar1976monotone}; given a vector $u_0 \in \Real^n$, \eqref{PPA} generates a sequence $\{u_k\}$ defined as follows:
\begin{align} \label{PPA}
0 \in T(u_{k+1}) + \frac{1}{\lambda}(u_{k+1}-u_k). 
\end{align} 
Consequently,  the resolvent operator of a monotone set-valued operator, denoted by $J_{\lambda}^T(u)$, is defined as $J_{\lambda}^T \triangleq (I + \lambda T)^{-1}.$ 
Then we may state the update rule \eqref{PPA} as an update using the resolvent operator, i.e.
\begin{align*} \left[ 0 = T(u_{k+1}) + {1\over \lambda} (u_{k+1}-u_k) \right] \iff \left[u_{k+1} = J^T_{\lambda}(u_k)\right]. \end{align*} 
Consider the generalized proximal point method, rooted in the seminal work by~\citeN{rockafellar1976monotone}, which employs the following update rule (see~\citeN{corman2014generalized}):
\begin{align}\label{GPPA}
u_{k+1} \triangleq \eta J_{\lambda}^T(u_k) + (1-\eta) u_k, 
\end{align}
which reduces to \eqref{PPA} when $\eta = 1$. We now define the Yosida approximation operator $T_{\lambda}$ as 
$T_{\lambda} \triangleq \tfrac{(I- J_{\lambda}^T)}{\lambda}.$ 
It has been shown that if $T$ is a monotone set-valued map and $J_{\lambda}^T$ denotes its resolvent, then $T_{\lambda}$ is $\lambda$-firmly non-expansive and ${1\over \lambda}$-Lipschitz continuous. Moreover, we have that from~\citeN[Prop.~2.4]{corman2014generalized} that for all $\lambda > 0$, $0 \in T(u) \Leftrightarrow T_{\lambda}(u) = 0. $
\us{In other words, computing a zero of a generalized equation is equivalent to computing a root of a suitably defined nonlinear equation.} By definition of $T_{\lambda}(x)$, we note that \eqref{GPPA} can be recast as follows:
\begin{align*}
u_{k+1} & = \eta J_{\lambda}^T(u_k) + (1-\eta) u_k= u_k - \eta (u_k - J_{\lambda}^T(u_k)) = u_k - \eta \lambda T_{\lambda}(u_k).
\end{align*}
\us{In other words, the update is seen to be similar to a subgradient scheme, where the subgradient of $f$ is replaced by the $T_{\lambda}(u_k)$ while the steplength is given by $\eta\lambda$.}
We now reproduce the main rate statements available for \eqref{GPPA} in both monotone and strongly monotone regimes (see~\cite{corman2014generalized}).
\begin{align}
	\|T_{\lambda}(u_k)\|^2 & \leq \mathcal{O}\left(\tfrac{1}{k+1} \right), \qquad \qquad k \geq 0  \tag{(Monotone), $\gamma \in (0,2)$} \\
 \|u_k - u^*\|^2 & \leq \rho^k \|u_0 - u^*\|^2, \qquad k \geq 0, \quad \rho \triangleq \left| 1- \tfrac{\gamma \lambda \alpha}{1+\lambda \alpha}\right|.    \tag{St.~Monotone), $\gamma \in (0,2)$} 
\end{align}
\subsection{A Stochastic Proximal Point Framework}\label{svi_prox}
Faced with an expectation-valued map $T(u) \triangleq
\mathbb{E}[T(u,\omega)]$, a direct application of 
\eqref{PPA} is impossible since it requires exactly evaluating
the resolvent $J_{\lambda}^T$ of an expectation-valued map.
Consequently, contending with such problems requires  
inexactly evaluating the resolvent $J_{\lambda}^T$; in effect, we
compute a noise-afflicted variant of $T_{\lambda}(u_k)$, given by $T_{\lambda}(u_k)+w_k$, where
 $ \mathbb{E}[\|w_k\|^2] \leq
\tfrac{\nu^2}{N_k}, $ utilizing update rule~\eqref{st-GPPA-T}. 
\begin{align}\label{st-GPPA-T}
u_{k+1} = u_k - \eta \lambda (T_{\lambda}(u_k)+w_k).
\end{align}
Our goal lies in inexactly computing $J_{\lambda}^T(u)$ when $T(u)$
corresponds to the generalized equation reformulation of a variational
inequality problem. Specifically, $x$ is a solution to \eqref{VI$(X,F)$} if and only
if  \begin{align*} 0 \in T(x) \triangleq F(x) + \Nscr_X(x), \end{align*} where
$\Nscr_X(x)$ refers to the normal cone of $X$ at $x$.  is defined as $T(u)
\triangleq F(u) + \Nscr_X(u)$, where $\Nscr_X(u)$ is normal cone of $X$. Recall
that when $F$ is a single-valued monotone map, $F(x) +\Nscr_X(x)$ is maximal
monotone (see~\cite{facchinei2007finite}), allowing for applying the proximal
point framework. Hence, 
\begin{align*}
z^*_k=J^T_\lambda(u_k) \iff 0 \in (F + \tfrac{1}{\lambda} {\bf I})(z^*_k) - {\tfrac{1}{\lambda} }u_k   + \Nscr_{X}(z^*_k)  \iff z^*_k \in \mbox{SOL}(X,F+\tfrac{1}{\lambda} {\bf I} -  {\tfrac{1}{\lambda}}u_k),
\end{align*}
\us{where SOL$(X,H + u)$ denotes the solution set of variational inequality
problem with set $X$ and mapping $F(x)+u$.} It can be immediately seen that if
$F$ is a monotone map, then $F+\tfrac{1}{\lambda} {\bf I}$ is a
$\tfrac{1}{\lambda}-$strongly monotone map. In this section, when $F$ is a monotone expectation-valued map, we develop a proximal framework of \eqref{VI}  \vvs{reliant on} solving a sequence of
strongly monotone stochastic variational inequality problems,
each of which is solved with increasing exactness. \uvs{Solving
\eqref{VI}  is equivalent to
resolving} $0\in T(x) \triangleq \mathbb{E}[F(x,\omega)] + \Nscr_X(x)$. \us{We employ a} stochastic
proximal point framework described in Section \ref{svi_prox}, formalized in 
Algorithm \ref{SPP}. We proceed to show that the scheme admits a convergence
rate of $\mathcal O(1/k)$ where at each step of
\eqref{st-GPPA-T}, we obtain an increasingly accurate solution of $J_\lambda^T(u_k)$ by running
Algorithm \ref{SVI alg.} for a finite (but increasing with $k$ number of iterations denoted by $\ell_k$.\\

  \begin{algorithm} [{ Proximal-point Algorithm with Variable Sample-sizes (PPAWSS)}]
\label{SPP}

\em
\begin{enumerate}
\item[]
\item[(2.0)] Given $u_0\in X$, $\lambda>0$ and let $\eta\in (0,2)$, $k = 0$; 
\item[(2.1)] Run {\bf VS-Ave($F+1/\lambda I-1/\lambda u_k,y_0,1/\lambda,L+1/\lambda,\ell_k$)};
\item[(2.2)] $z_{k}=\bar y_{\ell_k}$;
\item[(2.3)] $u_{k+1}=\eta z_k+(1-\eta)u_k$;
\item[(2.4)] If $k > K$, then stop; else $k = k+1$; return
to (1); 
\item[(2.5)] Return $u_K$.
\end{enumerate}
\end{algorithm}
In the following theorem, we state the convergence result for Algorithm \ref{SPP}. 
\begin{theorem}[{Convergence rate of Stochastic Proximal Point Scheme (PPAWSS)}]
Suppose $\{u_k\}$ is a sequence generated by Algorithm \ref{SPP}, $T$ is defined as $T(x) = \mathbb{E}[F(x,\omega)]+\mathcal{N}_X(x)$, and $T_\lambda$ denotes the Yosida approximation operator. Suppose Assumptions \ref{assump-F} and \ref{assump_error} hold. Let $\ell_k=\uvs{\lfloor 2\log_{1/q}((1+k)^\alpha)\rfloor}$ for any $\alpha>1$ and $e_k={1\over \lambda}(J_\lambda(x_k)-z_k)$. 

\noindent (i) For any $\lambda>0$ and $K>0$, the following holds:
\begin{align*}\us{\mathbb{E}[\|T_\lambda(u_K)\|^2]}\leq {\tilde C\over \eta(2-\eta)\lambda^2(K+1)}+{2C\eta\hat C\over (2-\eta)\lambda(K+1)}=\mathcal O(1/(K+1)),\end{align*}
where $C$ is defined in Theorem \ref{rate svi}, $\hat C\triangleq2+{1\over 2\alpha-1}+{1\over 2(2\alpha-1)(\alpha-1)}$, and 
\begin{align*}\tilde C\triangleq\|u_0-u^*\|^2+\af{6}C\eta^2\lambda\left(1+{1\over 2\alpha-1}\right)+2\eta \|u_0-u^*\|\sqrt{2C\lambda}\left(1+{1\over \alpha-1}\right)+\af{4}C\eta^2\lambda\left(1+{1\over \alpha-1}\right)^2.\end{align*}

\noindent (ii) \vvs{Suppose an $\epsilon$-solution $u_k$ is
defined as $\mathbb{E}[\|T(u_k)\|^2] \leq \epsilon$. Then
computing an $\epsilon$-solution requires solving $\mathcal
O\left(\tfrac{1}{\epsilon} \log
\left(\tfrac{1}{\epsilon}\right)\right)$ \vvs{proximal}
problems and  $\mathcal O(1/\epsilon)^{1+2\alpha\beta}$ evaluations of the map}. 

\end{theorem}
\begin{proof}
(i) Let $\tilde u_{k+1}  \ \triangleq u_k-\eta\lambda T_\lambda(u_k)$. By invoking the definition of $J_\lambda^T$, we have that:
\begin{align}\label{use J}
\langle T_\lambda(u_1)-T_\lambda(u_2),u_1-u_2\rangle\nonumber&=\langle T_\lambda(u_1)-T_\lambda(u_2),\lambda T_\lambda(u_1)-\lambda T_\lambda(u_2)\rangle+\langle T_\lambda(u_1)-T_\lambda(u_2),J_\lambda^T(u_1)-J_\lambda^T(u_2)\rangle\\&=\lambda\|T_\lambda(u_1)-T_\lambda(u_2)\|^2+\langle T_\lambda(u_1)-T_\lambda(u_2),J_\lambda^T(u_1)-J_\lambda^T(u_2)\rangle.
\end{align}
Using the definition of $\tilde u_{k+1}$, \eqref{use J} and by recalling that $T_\lambda(u^*)=0$, we get
\begin{align} \notag
\|\tilde u_{k+1}-u^*\|^2 & =\|\us{u}_k-u^*-\eta\lambda T_\lambda(u_k)\|^2=\|u_k-u^*\|^2+\eta^2\lambda^2\|T_\lambda(u_k)\|^2-2\eta\lambda\langle T_\lambda(u_k),u_k-u^*\rangle \\
\notag & =\|u_k-u^*\|^2-\eta(2-\eta)\lambda^2\|T_\lambda(u_k)\|^2-2\eta\lambda\underbrace{\langle T_\lambda(u_k),J_\lambda(u_k)-J_\lambda(u^*)\rangle}_{\tiny {\geq 0, T_{\lambda}(u_k) \in T(J_{\lambda}(u_k)), \ T \mbox{is monotone}}} \\
\label{bound tilde u}
& \leq\|u_k-u^*\|^2-\eta(2-\eta)\lambda^2\|T_\lambda(u_k)\|^2 \\
\implies  & \|u_{k+1}-u^*\| \leq \|\tilde u_{k+1}-u^*\|+\overbrace{\|u_{k+1}-\tilde u_{k+1}\|}^{\tiny  \ \leq \|\eta\lambda e_k\|}\leq \|u_0-u^*\|+\eta\lambda \sum_{i=0}^k \|e_i\|.   
\label{bd-u}
\end{align}
Since $\|\tilde u_{k+1}-u^*\|\leq \|u_0-u^*\|+\eta\lambda \sum_{i=0}^k\|e_i\|$, we have the following. 
\begin{align} \label{bd-u-2}
\|u_{k+1}-u^*&\|^2{\leq} \|\tilde u_{k+1}-u^*\|^2+\|u_{k+1}-\tilde u_{k+1}\|^2+2\|\tilde u_{k+1}-u^*\|\|u_{k+1}-\tilde u_{k+1}\|\\ & \overset{\eqref{bound tilde u},\eqref{bd-u}}{\leq} \|u_k-u^*\|^2-\eta(2-\eta)\lambda^2\|T_\lambda(u_k)\|^2+\eta^2\lambda^2\|e_k\|^2+2\eta\lambda \|e_k\|\left(\|u_0-u^*\|+\eta\lambda \sum_{i=0}^k \|e_i\|\right). \notag
\end{align}
Summing  \eqref{bd-u-2} from $k=0$ to $K$ and taking expectations, we have the following.
\begin{align}
\nonumber&\eta(2-\eta)\lambda^2\sum_{k=0}^K\mathbb \us{[\|T_\lambda(u_k)\|]}\\ \nonumber&\leq \|u_0-u^*\|^2-\|u_{K+1}-u^*\|^2+\sum_{k=0}^K \left(\eta^2\lambda^2\|e_k\|^2+2\eta\lambda  \|e_k\|\left(\|u_0-u^*\|+\eta\lambda \sum_{i=0}^k \|e_i\|\right) \right)\\
\nonumber&\eta(2-\eta)\lambda^2\sum_{k=0}^K\mathbb E[\|T_\lambda(u_k)\|]\\ \nonumber&\leq \|u_0-u^*\|^2+ \eta^2\lambda^2\left(\sum_{k=0}^K\mathbb E\left[\|e_k\|^2\right]\right)+2\eta\lambda \|u_0-u^*\|\left(\sum_{k=0}^K \mathbb E\left[\|e_k\|\right]\right)+2\eta^2\lambda^2 \mathbb E\left[\left(\sum_{i=0}^K \sum_{j=0}^{i} \|e_j\|  \| \|e_i\|\right)\right] .
\end{align}
\vvs{If} $\hat{\mathcal F}_j \vvs{ \ \triangleq \ } \{u_0,\hdots,u_j\}$, then from the law of total expectation, we know that for all $j<i$, $\mathbb E\left[\|e_i\| \|e_j\|\right]=\mathbb E\left[\mathbb E\left[\|e_i\| \|e_j\|\mid \hat{\mathcal F_j}\right]\right]=\mathbb E\left[\|e_j\|\mathbb E\left[\|e_i\| \mid \hat{\mathcal F_j}\right]\right]$. Hence the following holds:
\begin{align}\label{bound T}
& \eta(2-\eta)\lambda^2\sum_{k=0}^K\mathbb E[\|T_\lambda(u_k)\|]\nonumber\leq \|u_0-u^*\|^2+ \eta^2\lambda^2\left(\sum_{k=0}^K\mathbb E[\|e_k\|^2]\right)+2\eta\lambda \|u_0-u^*\|\left(\sum_{k=0}^K\mathbb E[\|e_k\|]\right)\\&
+{{2\eta^2\lambda^2 \mathbb E\left[\sum_{k=0}^K \|e_k\|^2\right]}}+2\eta^2\lambda^2 \sum_{i=0}^K\sum_{j=0}^{i-1}\mathbb E\left[\|e_j\|\mathbb E\left[\|e_i\| \mid \hat{\mathcal F_j}\right]\right]\\& =\|u_0-u^*\|^2+ 3\eta^2\lambda^2\left(\sum_{k=0}^K\mathbb E[\|e_k\|^2]\right)+2\eta\lambda \|u_0-u^*\|\left(\sum_{k=0}^K\mathbb E[\|e_k\|]\right)
+{2\eta^2\lambda^2} \sum_{i=0}^K\sum_{j=0}^{i-1}\mathbb E\left[\|e_j\|\mathbb E\left[\|e_i\| \mid \hat{\mathcal F_j}\right]\right].\nonumber
\end{align}
From Theorem \ref{rate svi}, $\mathbb E[\|e_k\|^2]\leq {2C\over
\lambda} q^{\ell_k}  \leq
{2C\over\lambda}(k+1)^{-2\alpha}$, since $\ell_k \triangleq \lfloor 2\log_{1/q}((1+k)^\alpha)\rfloor$ for any
$\alpha>1$. Therefore, 
\begin{align}\label{bound int}
\sum_{k=0}^K {2C\over\lambda}(k+1)^{-2\alpha}\leq \tfrac{2C}{\lambda}\left(1+\int_0^K (x+1)^{-2\alpha}dx\right)=\tfrac{2C}{\lambda}\left(1+{(K+1)^{1-2\alpha}-1\over 1-2\alpha}\right)\leq \tfrac{2C}{\lambda}\left(1+{1\over 2\alpha-1}\right).
\end{align}
Now {by substituting} \eqref{bound int} in \eqref{bound T} and by Jensen's inequality, we obtain the following. 
\begin{align}\label{constant bound}
\eta(2-\eta)\lambda^2\sum_{k=0}^K\mathbb E[\|T_\lambda(u_k)\|]&\nonumber\leq \|u_0-u^*\|^2+6C\eta^2\lambda\left(1+{1\over 2\alpha-1}\right)+2\eta \|u_0-u^*\|\sqrt{2C\lambda}\left(1+{1\over 2\alpha-1}\right)\\&+4C\eta^2\lambda\left(1+{1\over 2\alpha-1}\right)^2 \triangleq \tilde C \implies  \eta(2-\eta)\lambda^2\sum_{k=0}^K\mathbb E[\|T_\lambda(u_k)\|]\leq \tilde C.
\end{align}
Since $\|T_\lambda(u_{k+1})\| = \|T_\lambda(u_{k+1})-T_\lambda(u_{k})+ T_\lambda(u_{k})\|$, the following holds:
\begin{align*}
  \ \|T_\lambda(u_{k+1})\|^2 & \uvs{ \ = \ } \|T_\lambda (u_k)\|^2+\|T_\lambda(u_{k+1})-T_\lambda(u_k)\|^2+2\left(T_\lambda(u_{k+1})-T_\lambda(u_k)\right)^TT_\lambda(u_k)\\&
\leq  \|T_\lambda (u_k)\|^2-{2-\eta\over \eta}\|T_\lambda(u_{k+1})-T_\lambda(u_k)\|^2+{2\over \eta}\|T_\lambda(u_{k+1})-T_\lambda(u_k)\| \|\lambda e_k\|\\&
\leq \|T_\lambda (u_k)\|^2-{2-\eta\over \eta}\|T_\lambda(u_{k+1})-T_\lambda(u_k)\|^2+{2-\eta\over \eta}\|T_\lambda(u_{k+1})-T_\lambda(u_k)\|^2+{\eta\over 2-\eta}\|e_k\|^2 \\
& \leq \|T_\lambda (u_k)\|^2+{\eta\over 2-\eta}\|e_k\|^2,
\end{align*}
where the first inequality arises from adding and subtracting $\tfrac{2}{\eta} \|T_{\lambda}(u_{k+1})-T_{\lambda}(u_k)\|^2$ and the Cauchy-Schwarz inequality \vvs{while} the \uvs{second} inequality \us{follows from Young's} inequality. Summing from $k=n$ to $K$, $\mathbb E[\|T_\lambda(u_K)\|^2]\leq \mathbb E[\|T_\lambda(u_n)\|^2]+{\eta\over 2-\eta}\sum_{k=n}^K\mathbb E[\|e_k\|^2]$. Using inequality \eqref{constant bound} we obtain the following
\begin{align*}\tilde C\geq \eta(2-\eta)\lambda^2 \sum_{k=0}^K\mathbb E[\|T_\lambda(u_k)\|]\geq \eta(2-\eta)\lambda^2\left((K+1)\mathbb E[\|T_\lambda(u_K)\|^2]-{\eta\over (2-\eta)}\sum_{k=0}^K\sum_{j=k}^{K-1}\mathbb E[\|e_j\|^2].\right),\\
\implies \mathbb E[\|T_\lambda (u_K)\|^2]\leq {\tilde C\over \eta(2-\eta)\lambda^2(K+1)}+{\eta\over (2-\eta)(K+1)}\sum_{k=0}^K\sum_{j=k}^{K-1}\mathbb E[\|e_j\|^2].\end{align*}
Using the fact that $\mathbb E[\|e_k\|^2\leq {2C\over\lambda}(k+1)^{-2\alpha}$, the following can be obtained:
\begin{align*}\mathbb E[\|T_\lambda (u_K)\|^2]\leq {\tilde C\over \eta(2-\eta)\lambda^2(K+1)}+{2C\eta\over (2-\eta)\lambda(K+1)}\sum_{k=0}^K\sum_{j=k}^{K-1}(j+1)^{-2\alpha}.\end{align*}
Note that  for any $\alpha>1$ we have that:
\begin{align*}\sum_{k=0}^K\sum_{j=k}^{K-1}(j+1)^{-2\alpha}&\leq \int_0^{K+1}\int_0^{K+1}(x+1)^{-2\alpha}dx dy\leq \sum_{k=0}^{K-1}\left((k+1)^{-2\alpha}+\int_k^{K-1}(x+1)^{-2\alpha}dx\right)\\&
=\sum_{k=1}^{K-1}(k+1)^{-2\alpha}+{1\over 2\alpha-1}\left((k+1)^{1-2\alpha}-K^{1-2\alpha}\right)\\&\leq 1+\int_0^{K-1} (x+1)^{-2\alpha}dx+{1\over 2\alpha-1}\left(1+\int_0^{K-1}(x+1)^{1-2\alpha}dx\right)
\\&\leq 2+{1\over 2\alpha-1}+{1\over 2(2\alpha-1)(\alpha-1)}
\implies \mathbb{E}[\|T_\lambda(u_K)\|^2] \leq \mathcal O(1/(K+1)).
\end{align*}

\noindent (ii) We showed that there exists a constant $\bar C$ where $\|T_\lambda(u_K)\|^2\leq \bar C/(K+1)$. Hence the minimum number of steps to obtain $\epsilon$-solution can be computed as follows for $\rho \triangleq q^\beta$ where $\beta>1$:
\begin{align*}
\sum_{k=0}^{\bar C/\epsilon}\ell_k \leq \sum_{k=0}^{\bar C/\epsilon} 2\log_{1/q}(1+k)^{\alpha}\leq \int_0^{\bar C/\epsilon+1}{2\alpha\over \log 1/q}\log(1+x)dx \uvs{\leq} {2\alpha\over \log 1/q}\left((\bar C/\epsilon+2)\log (\bar C/\epsilon+2)\right).
\end{align*}
The oracle complexity in terms of the number of evaluations of $G(x,\xi)$ can then be bounded as follows:
\begin{align*}
\sum_{k=0}^{\bar C/\epsilon+1}\sum_{s=0}^{\ell_k} \lfloor \rho^{-s}\rfloor&= \sum_{k=0}^{\bar C/\epsilon+1} \frac{(1/\rho)^{\ell_k+1}-1}{1/\rho-1}\leq \sum_{k=0}^{\bar C/\epsilon+1} \frac{(1+k)^{2\alpha\beta}}{1-\rho}\leq {\int_0^{\bar C/\epsilon+2}(1+x)^{2\alpha\beta}dx \over 1-\rho}
={(2+\bar C/\epsilon)^{1+2\alpha\beta}\over (1-\rho)(1+2\alpha\beta)}. 
\end{align*}
\end{proof}
\section{NUMERICAL RESULTS}
In this section, we compare our scheme with the extragradient scheme presented by \shortciteN{iusem2017extragradient} on  the following stochastic bimatrix game problem:
\begin{align}\label{example1}
\min_{x\in \Delta_n}\ \max_{y\in \Delta_m} \quad \af{\mathcal L}(x,y) \ \triangleq\  \mathbb E[\langle A(\xi)x,y\rangle],
\end{align} where $\Delta_n$ is an $n$-dimensional simplex. We may recast this saddle
point problem \eqref{example1} as an SVI to find $z=\uvs{(x,y)}\in \Delta_n\times
\Delta_m$ such that $\begin{bmatrix}A(\xi)^Ty\\-A(\xi)x \end{bmatrix}^T(z-\bar
z)\geq 0$, for all $\bar z\in \Delta_n\times \Delta_m$. Suppose the total
simulation budget is $1e7$, $n=20$, and $m=10$. \vvs{Table~\ref{tab1}} shows that for
different choices of Lipschitz constants (L), our scheme compares well with the
extragradient scheme proposed by \shortciteN{iusem2017extragradient}.  Note that we
computed the optimal solution \vvs{in} \eqref{example1} \vvs{by solving the sample-average problem using $1e6$ samples via \texttt{cvx}}. Steplength and sample sizes have been chosen as
suggested by \shortciteN{iusem2017extragradient}, $\eta<1/(\sqrt 6L)$, and $N_k=\lceil
\theta (k+\mu) \ln(k+\mu)^{1+b}\rceil$, \af{where $\theta=1$, $b=10^{-3}$ and $\mu=2+10^{-3}$.} In the 
({\bf PPAWSS}) scheme, we set $\eta=1$ and $N_k=\lfloor \rho^{-k}\rfloor$, \af{where $\rho=q^{-\beta}$, $q=1-(1/(\kappa+1))$ and $\beta=1.001$}.\begin{table}[htb]
\centering
\scriptsize
\caption{Comparing ({\bf PPAWSS}) scheme with extragradient method for SVI.\label{tab1}}
\begin{tabular}{|c||c|c||c|}
\hline
&&\multicolumn{2}{|c||}{$|\af{\mathcal L}(x_K,y_K)-\af{\mathcal L}(x^*,y^*)|$} \\\hline
Lipschitz constant (L)&$\lambda$&  PPAWSS&Extragradient\\ \hline\hline
7.05&3500&6.4576e-05&1.2697e-04\\ \hline
70.5&1200& 3.5218e-04 & 7.1742e-04\\\hline
705&40&2.5911e-03&6.0048e-03\\ \hline
\end{tabular}
\end{table}

It can be seen in Table \ref{tab1}, when $L$ increases, a smaller $\lambda$ tends to perform better since the subproblem in line (2) of ({\bf PPAWSS}) is $\tfrac{1}{\lambda}$-strongly monotone with condition number $\lambda(L+1/\lambda)$; therefore to manage the condition number,  we reduce $\lambda$. 
\section{CONCLUDING REMARKS} We developed two variable
sample-size methods to resolve strongly monotone and monotone
SVIs, \vvs{respectively}. In the former setting, a linear convergence rate and near
optimal oracle complexity \vvs{are} obtained with a more modest
dependence on the condition number for an averaging-based
variance-reduced scheme ({\bf VS-Ave}). In monotone regimes,  we
develop amongst the first \uvs{proximal point algorithms with variable sample-sizes ({\bf PPAWSS})} where \vvs{strongly monotone} subproblems are solved with increasing inexactness via
({\bf VS-Ave}), achieving the canonical deterministic rate.
Finally, numerics suggest that proposed scheme compares well
with its extragradient competitor. 

\footnotesize
\bibliographystyle{wsc}
\bibliography{demobib.bib}

\section*{AUTHOR BIOGRAPHIES}

\noindent {\bf AFROOZ JALILZADEH} is a Ph.D. candidate in the Harold and Inge Marcus Department of Industrial and Manufacturing Engineering at Pennsylvania State University. She received her bachelor degree in Mathematics and Applications from University of Tehran in Iran. Her research interests include stochastic optimization, variational inequality \vvs{problems}, convex optimization, and \vvs{m}achine learning. Her email address is \email{azj5286@psu.edu}. 
\\

\noindent {\bf UDAY V. SHANBHAG} is a professor  in the Harold and Inge Marcus Department of Industrial and Manufacturing Engineering at Pennsylvania State University. He holds a Ph.D. in Management science and Engineering from Stanford University (2006). His research interest lies in stochastic and nonlinear optimization, nonsmooth analysis and variational inequality problems. His email address is \email{udaybag@psu.edu}. 
\end{document}